\documentclass[12pt,a4paper,english]{article}
\usepackage{amsmath,amsfonts,amsthm,amssymb}
\usepackage{paralist}
\usepackage{enumerate}
\usepackage{babel}
\author{Jeroen Demeyer\thanks{Postdoctoral Fellow of the Research Foundation --- Flanders (FWO).
\textbf{Address:} Ghent University, Department of Mathematics, Kr{\ij}gs\-laan 281, 9000 Gent, Belgium.
} \and Antonella Perucca\thanks{Postdoctoral Fellow of the Research Foundation --- Flanders (FWO).
\textbf{Address:} K.U.Leuven, Department of Mathematics, Celest{\ij}nenlaan 200\,B, 3001 Leuven, Belgium.
}}
\title{The constant of the support problem for abelian varieties}
\providecommand{\theoremname}{Theorem}
\providecommand{\maintheoremname}{Main Theorem}
\providecommand{\conjname}{Conjecture}
\providecommand{\openname}{Open Problem}
\providecommand{\questionname}{Question}
\providecommand{\propname}{Proposition}
\providecommand{\obsname}{Observation}
\providecommand{\lemmaname}{Lemma}
\providecommand{\mainlemmaname}{Main Lemma}
\providecommand{\corname}{Corollary}
\providecommand{\claimname}{Claim}
\providecommand{\definename}{Definition}
\providecommand{\examplename}{Example}
\providecommand{\cexamplename}{Counterexample}
\providecommand{\exercisename}{Exercise}
\providecommand{\problemname}{Problem}
\providecommand{\factname}{Fact}
\providecommand{\factsname}{Facts}
\providecommand{\remarkname}{Remark}
\providecommand{\solutionname}{Solution}
\providecommand{\stepname}{Step}

\providecommand{\condname}{Condition}
\theoremstyle{plain}
\newtheorem{theorem}{\theoremname}[section]
\newtheorem*{theorem*}{\theoremname}

\newtheorem*{mtheorem*}{\maintheoremname}

\newtheorem*{conj*}{\conjname}

\newtheorem*{open*}{\openname}
\newtheorem{prop}[theorem]{\propname}
\newtheorem*{prop*}{\propname}
\newtheorem{obs}[theorem]{\obsname}
\newtheorem*{obs*}{\obsname}
\newtheorem{lemma}[theorem]{\lemmaname}
\newtheorem*{lemma*}{\lemmaname}

\newtheorem*{mlemma*}{\mainlemmaname}
\newtheorem{cor}[theorem]{\corname}
\newtheorem*{cor*}{\corname}

\theoremstyle{definition}
\newtheorem{define}[theorem]{\definename}
\newtheorem*{define*}{\definename}

\newtheorem*{example*}{\examplename}

\newtheorem*{cexample*}{\cexamplename}
\newtheorem{question}[theorem]{\questionname}
\newtheorem*{question*}{\questionname}

\newtheorem*{exercise*}{\exercisename}

\newtheorem*{problem*}{\problemname}

\newtheorem*{fact*}{\factname}

\newtheorem*{facts*}{\factsname}

\newtheorem*{cond*}{\condname}

\theoremstyle{remark}

\newcounter{step}

\newcounter{case}
\newcommand{\inv}{^{-1}}

\newcommand{\mi}{\mathrm{i}}

\newcommand{\frakm}{\mathfrak{m}}
\newcommand{\frakp}{\mathfrak{p}}

\newcommand{\C}{\mathbb{C}}

\newcommand{\N}{\mathbb{N}}
\newcommand{\Q}{\mathbb{Q}}
\newcommand{\R}{\mathbb{R}}
\newcommand{\Z}{\mathbb{Z}}
\newcommand{\calA}{\mathcal{A}}
\newcommand{\calB}{\mathcal{B}}

\newcommand{\calF}{\mathcal{F}}

\newcommand{\calM}{\mathcal{M}}
\newcommand{\calN}{\mathcal{N}}
\newcommand{\calO}{\mathcal{O}}

\newcommand{\AS}[2][void]{\ifthenelse{\equal{#1}{void}}{\mathbb{A}\!^{#2}}{\mathbb{A}\!^{#2}(#1)}}
\newcommand{\PS}[2][void]{\ifthenelse{\equal{#1}{void}}{\mathbb{P}^{#2}}{\mathbb{P}^{#2}(#1)}}

\newcommand{\FF}[1]{\mathbb{F}\!_{#1}}
\DeclareMathOperator{\Ann}{Ann}

\DeclareMathOperator{\End}{End}
\DeclareMathOperator{\Hom}{Hom}

\DeclareMathOperator{\gal}{Gal}

\DeclareMathOperator{\ord}{ord}

\newcommand{\set}[2]{\{{#1}\mid{#2}\}}

\newcommand{\form}[1]{\langle{#1}\rangle}

\newcommand{\vece}{\mathbf{e}}
\newcommand{\vecr}{\mathbf{r}}
\newcommand{\vecs}{\mathbf{s}}

\newcommand{\vecx}{\mathbf{x}}
\newcommand{\vecy}{\mathbf{y}}
\begin{document}
%###################################################################################################

\maketitle
%===================================================================================================

\begin{abstract}
Let $A$ be an abelian variety defined over a number field $K$
and let $P$ and $Q$ be points in $A(K)$ satisfying the following condition:
for all but finitely many primes $\mathfrak p$ of $K$,
the order of $(Q \bmod \mathfrak p)$ divides the order of $(P \bmod \mathfrak p)$.
Larsen proved that there exists a positive integer $c$ such that $c Q$ is in the $\End_K(A)$-module generated by $P$.
We study the minimal value of $c$ and construct some refined counterexamples.
%The minimal value of $c$ can be bounded by a constant which depends only on $A$ and $K$.
\end{abstract}
%---------------------------------------------------------------------------------------------------


\section{Introduction}
%---------------------------------------------------------------------------------------------------

Let $A$ be an abelian variety defined over a number field $K$. Let $P$, $Q$ be points in $A(K)$ satisfying the following condition:
 for all but finitely many primes $\mathfrak p$ of $K$, the order of $(Q \bmod \mathfrak p)$ divides the order of $(P \bmod \mathfrak p)$.
The support problem asks whether there exists a $K$-endomorphism of $A$ mapping $P$ to $Q$.

If $A$ is $K$-simple and the points $P$ and $Q$ have infinite order,
Khare and Prasad proved in \cite[Theorem 1]{KharePrasad} that indeed $\phi(P)=Q$ for some $\phi$ in $\End_K (A)$.
%For elliptic curves, this result is due to Corrales-Rodrig\'a\~nez and Schoof (\cite[Theorem 2]{cr-schoof}).
This result does not hold for general abelian varieties. However, Larsen proved that there exist a $K$-endomorphism $\phi$ of $A$
and a positive integer $c$ such that $\phi(P)=cQ$ (\cite[Theorem 1]{larsen-supp-av}).
So in general one cannot take $c=1$ (not even if $\phi$ is taken in $\End_{\bar{K}} (A)$),
as shown by Larsen in \cite[Proposition 2]{larsen-supp-av}).
%it does not help to require the condition for all primes $\mathfrak p$ of $K$ (\cite[Proposition 3]{larsen-supp-av}).

We study the minimal positive integer $c$ (depending only on $A$ and $K$) for which the following holds:
for every pair of points $P$, $Q$ in $A(K)$ satisfying the condition of the support problem,
there exists a $K$-endomorphism $\phi$ of $A$ such that $\phi(P)=cQ$.
It is known that such an integer exists
(\cite[Proposition 10]{Perucca2} or \cite[Proposition 4.3 and Theorem 5.2]{Larsenwhitehead}):
we call it the constant of the support problem.
The following question arises:

\begin{question}\label{cimin}
Does the constant of the support problem divide the exponent of the torsion part of $A(K)$?
\end{question}

The answer is affirmative for simple abelian varieties, as a consequence of \cite[Theorem 1]{KharePrasad}.
Larsen proved in \cite[Proposition 4.3 and Theorem 5.2]{Larsenwhitehead}
that the answer is affirmative whenever all the Tate modules of $A$ are integrally semi-simple (\cite[Definition 4.1]{Larsenwhitehead}).

In this paper, we use a new method to study the support problem:
we view the Mordell--Weil group as a module over the endomorphism ring
and apply the theory of maximal orders in division algebras.
%In this paper, we use endomorphism rings of abelian varieties to
%study Question~\ref{cimin}.
%In the past, the support problem has been attacked
%using methods from algebraic number theory and algebraic geometry.
%Our approach instead is viewing the group of rational points of the abelian
%variety as a module over the endomorphism ring and then applying the theory of
%maximal orders in division algebras.

We prove that the answer to Question~\ref{cimin} is affirmative
whenever $A$ is a power of a simple abelian variety $A_1$
such that $\End_K(A_1)$ is a maximal order in a division algebra.
More generally, the answer is affirmative
for products $\prod A_i^{e_i}$ of such powers, provided that $\Hom_K(A_i, A_j) = 0$ for $i \neq j$,
see Theorem~\ref{supp-tf-max}.
In particular, the answer is affirmative for at least one variety in every $K$-isogeny class
(this also follows from the results of Larsen in \cite{Larsenwhitehead}).

We also construct two counterexamples to Question~\ref{cimin} in section~\ref{examples}.
They are respectively of the following kind:
the power of a simple abelian variety whose endomorphism ring is not a maximal order;
an abelian variety which is $\bar{K}$-isomorphic (but not $K$-isomorphic)
to the power of an elliptic curve whose endomorphism ring is a maximal order.

With a similar construction, we answer in the negative to the question of the support problem for tori, see section \ref{exampleAtori}.

%A complete list of references for the support problem is the following: \cite{cr-schoof}, \cite{BGKjacobians}, \cite{Kharegalois}, \cite{larsen-supp-av}, \cite{KharePrasad}, \cite{Larsenwhitehead}, \cite{LarsenSchoof}, \cite{Wittenberg}, \cite{Baranczuk06}, \cite{GajdaGornisiewicz}, \cite{Perucca2}, \cite{Baranczuk08}, \cite{Perucca4}.

A motivation to study the support problem is given by the following theorem
which is a consequence of 
results on the support problem by Larsen (\cite{larsen-supp-av}), Khare and Prasad (\cite{KharePrasad})
and the second author (\cite{Perucca2}):

\begin{theorem}\label{classy}
Let $A$ be an abelian variety defined over a number field $K$.
Let $R$ be a point in $A(K)$ such that $\mathbb Z R$ is Zariski-dense in $A$.
Let $S$ be a set of primes of $K$ of Dirichlet density $1$.
\begin{enumerate}
\item The sequence
$$\{\ord(R \bmod \mathfrak p)\}_{\mathfrak p\in S}$$
determines the $K$-isomorphism class of $A$ and determines $R$ up to $K$-isomorphism.
\item Let $\ell$ be a prime number and write $\ord_\ell$ for the $\ell$-adic valuation of the order.
The sequence
$$\{\ord_\ell(R \bmod \mathfrak p)\}_{\mathfrak p\in S}$$
determines the $K$-isogeny class of $A$.
\end{enumerate}
\end{theorem}

We prove this result at the end of section~\ref{sec-condition}.
An important special case is when $A$ is $K$-simple
because then $\Z R$ is Zariski-dense in $A$ for any point $R$ of infinite order.
Notice that we had to assume that $\Z R$ is Zariski-dense in $A$: for example the point $R$ in $A$ and the point $(R,0)$ in the square of $A$ give rise to the same sequences.

\textbf{Acknowledgements.}
The authors would like to thank Jan Van Geel for giving very helpful advice
concerning the theory of maximal orders and Marc Hindry for useful explanations and comments.
The authors are also indebted to Arno Fehm for providing a proof of Proposition~\ref{arno},
which permitted to get rid of a technical condition in the main result.
%===================================================================================================

\section{A result on maximal orders}
%---------------------------------------------------------------------------------------------------

We begin by recalling some notions concerning algebras, modules and orders.
We mainly refer to \cite{reiner-maxord}.

Let $R$ be an associative ring with $1$, not necessarily commutative
and without zero divisors.
By ``$R$-module'', if not specified otherwise, we mean \emph{left} $R$-module.

Let $\calM$ be an $R$-module.
We say that $\calM$ is \emph{torsion-free}
if for all $\alpha \in R \setminus \{0\}$ and $P \in \calM \setminus \{0\}$,
we have $\alpha P \neq 0$.
We say that $\calM$ is \emph{divisible} if,
for every $P \in \calM$ and every $\alpha \in R \setminus \{0\}$,
there exists $Q \in \calM$ such that $P = \alpha Q$.

\begin{define}
Let $G$ be an $R$-module.
Let $\calM$ be a torsion-free submodule of $G$ and let $P \in G$.
We say that $P$ is \emph{independent} from $\calM$ if $\alpha P \notin \calM$
for all $\alpha \in R \setminus \{0\}$.
\end{define}

\begin{lemma}\label{infinite-rank-equiv}
Let $G$ be an $R$-module.
The following are equivalent:
\begin{compactenum}
\item\label{infinite-rank-equiv1} $G$ contains a free $R$-module of infinite rank.
\item\label{infinite-rank-equiv2} For all finitely generated $R$-modules $\calM \subseteq G$,
	there exists some $P \in G$ which is independent from $\calM$.
\end{compactenum}
\end{lemma}

\begin{proof}
\textbf{\ref{infinite-rank-equiv1} $\Rightarrow$ \ref{infinite-rank-equiv2}}:
Let $n \in \N$ be such that $\calM$ can be generated by $n$ elements.
By assumption, $G$ contains a free submodule $\calB = R B_1 \oplus \dots \oplus R B_{n+1}$.
Suppose that none of the points $B_i$ is independent from $\calM$.
Then there would exist $\alpha_i \in R\setminus \{0\}$ such that $\alpha_i B_i \in \calM$ for all $i = 1, \dots, n+1$.
Since $\calM$ is generated by $n$ elements,
there must be some non-trivial linear combination $\sum \beta_i (\alpha_i B_i)$ which is zero.
This is a contradiction.

\textbf{\ref{infinite-rank-equiv2} $\Rightarrow$ \ref{infinite-rank-equiv1}}:
We reason by induction.
Clearly, $\{0\} \subseteq G$ is free of rank $0$.
Let $\calM \subseteq G$ be a free $R$-module of rank $n$.
Since $\calM$ is finitely generated, there exists a $P \in G$
which is independent from $\calM$.
Then $\calM + R P \simeq R^{n+1}$.
Indeed, suppose that $Q + \alpha P = 0$ for some $Q \in \calM$ and $\alpha \in R$.
Then $\alpha P \in \calM$, therefore $\alpha = 0$ and also $Q = 0$.
\end{proof}

We recall the definition of tensor products for modules
over a ring which is not necessarily commutative:
\begin{define*}
Let $\calM$ be a right $R$-module and $\calN$ a left $R$-module.
Then the \emph{tensor product} $\calM \otimes_R \calN$
is the free abelian group on the symbols $m \otimes n$,
where $m \in \calM$ and $n \in \calN$,
modulo the relations
$(m + m') \otimes n = m \otimes n + m' \otimes n$,
$m \otimes (n + n') = m \otimes n + m \otimes n'$,
$(mr) \otimes n = m \otimes (rn)$
for all $m, m' \in \calM$, $n, n' \in \calN$, $r \in R$.

This tensor product is always an abelian group,
but in general not an $R$-module.
If $\calM$ is a two-sided $R$-module, then $\calM \otimes_R \calN$
becomes a left $R$-module by defining $r(m \otimes n) := (rm) \otimes n$.
\end{define*}

In this paper, a $\Q$-algebra means a ring $D \supseteq \Q$ which is a finite dimensional $\Q$-vector space.
We do not assume that the centre is exactly $\Q$.

Let $D$ be a $\Q$-algebra.
An \emph{order} in $D$ is a subring $R \subseteq D$ whose additive group is finitely generated
and such that $\Q R = D$.
A \emph{maximal order} is an order which is not contained in any larger order.
If $D$ is a number field, the ring of integers is the unique maximal order.

\begin{prop}
Let $R$ be a maximal order in a $\Q$-division algebra $D$.
The centre of $R$ is the ring of integers of the number field $K$,
where $K$ denotes the centre of $D$.
\end{prop}

\begin{proof}
Let $\calO_K$ denote the ring of integers of $K$.
We have $\Q R = D$, therefore the centre of $R$ is $R \cap K$.
Since $R \cap K$ is an order in $K$, we must have $R \cap K \subseteq \calO_K$.
Conversely, $\calO_K R$ is an order in $D$.
Since $R$ is a maximal order, this implies that $\calO_K \subseteq R$.
We conclude that $\calO_K = R \cap K$.
\end{proof}

\begin{lemma}\label{projective}
Let $D$ be a $\Q$-division algebra and let $R$ be a maximal order in $D$.
Let $\calM$ be a finitely generated and torsion-free $R$-module.
Then $\calM$ is projective.
\end{lemma}

\begin{proof}
Since $R$ is a maximal $\Z$-order in the $\Q$-algebra $D$,
it follows from \cite[(21.4)]{reiner-maxord} that $R$ is a left (and right) hereditary ring.
Such rings have the property
that all submodules of free modules are projective,
see \cite[(2.44)]{reiner-maxord}.
So it suffices to show that $\calM$ can be embedded in a free $R$-module.

Define $V := D \otimes_R \calM$.
Since $\calM$ is torsion-free, the map
$$
	\theta: \calM \to D \otimes_R \calM; m \mapsto 1 \otimes m
$$ is an embedding of $R$-modules.
Let $\{v_1, \dots, v_r\}$ be a basis of $V$ as $D$-vector space.
Since $\Q R = D$, there exists $b \in \Z \setminus \{0\}$ such that $b \theta(\calM)$
is contained in $R v_1 \oplus \dots \oplus R v_r$.
Then the map $\calM \to V: m \mapsto b \theta(m)$ embeds $\calM$ into $R^r$.
\end{proof}

\begin{theorem}\label{free-module}
Let $D$ be a $\Q$-division algebra and let $R$ be a maximal order in $D$.
Let $G$ be an $R$-module containing a submodule isomorphic to $R^\N$.
Let $\calM \subseteq \calN$ be finitely generated and torsion-free submodules of $G$.
Then there exists a finitely generated free module $\calF \subseteq G$
such that $\calF \cap \calN = \calM$.
\end{theorem}

\begin{proof}
By Lemma~\ref{projective}, $\calM$ is projective.
This means that there exists an abstract $R$-module $\calA$ such that $\calM \oplus \calA \simeq R^r$
for some $r \geq 0$.

By Lemma~\ref{infinite-rank-equiv},
there exists $B_1 \in G$ which is independent from $\calN$.
Since $\calN$ and $R B_1$ are torsion-free, also $\calN \oplus R B_1$ is torsion-free.
Analogously, we can find $B_2, \dots, B_r$ in $G$ such that, for all $k = 2, \dots, r$,
the point $B_k$ is independent from $\calN \oplus \form{B_1, \ldots, B_{k-1}}$.
Eventually, we get a finitely generated and torsion-free module $\calN \oplus \form{B_1, \ldots, B_r}$.

Since $\calM \oplus \calA \simeq R^r \simeq \form{B_1, \dots, B_r}$,
we can see $\calA$ as a submodule of $\form{B_1, \dots, B_r}$.
Now $\calM$ and $\calA$ are submodules of $G$ satisfying
$\calM \cap \calA \subseteq \calN \cap \form{B_1, \dots, B_r} = \{0\}$.
Let $\calF := \calM \oplus \calA$.
We clearly have $\calM \subseteq \calF \cap \calN$.

Let $P \in \calF \cap \calN$.
We need to show that $P \in \calM$.
We can write $P = P_\calM + P_\calA$ with $P_\calM \in \calM$ and $P_\calA \in \calA$.
Since $P \in \calN$ and $P_\calM \in \calN$, we also have $P_\calA \in \calN$.
But $\calN \cap \calA = \{0\}$, therefore $P_\calA = 0$ and $P = P_\calM$.
\end{proof}
%===================================================================================================

\section{Preliminaries on abelian varieties}
%---------------------------------------------------------------------------------------------------

Let $A$ be an abelian variety defined over a number field $K$
and let $L$ be an extension of $K$.
We write $\End_L(A)$ for the ring of endomorphisms of $A$ which are defined over $L$.
Let $D := \End_K(A) \otimes_\Z \Q$.
Then $D$ is a finite dimensional $\Q$-algebra and $\End_K(A)$ is an order in $D$.
If $A$ is $K$-simple, then $\End_K(A)$ does not contain any zero divisors
and $D$ is a division algebra.

The group $A(\bar{K})$ is a divisible $\Z$-module (\cite[Theorem 7.2]{mumford}).
If $A$ is $K$-simple, then $A(\bar{K})$
is also a divisible $\End_K(A)$-module:
this is because every non-zero element of $\End_K(A)$ is an isogeny and thus it 
divides the multiplication by some non-zero integer.

\begin{define}
We say that a point in $A(K)$ of infinite order is \emph{independent}
if it generates a free $\End_K(A)$-module or, equivalently,
a free $\End_{\bar{K}}(A)$-module.
This is also equivalent to the fact that $\mathbb ZR$ is Zariski-dense in $A$.
See \cite[Section 2]{Perucca1}.
We say that finitely many points $\{P_1, \ldots, P_n\}$ on $n$ abelian varieties
$A_1, \ldots, A_n$ are independent if the point $(P_1,\ldots, P_n)$
in $\prod_{i=1}^n A_i(K)$ is independent.
\end{define}

\begin{prop}\label{arno}
Let $K$ be a number field and fix an algebraic closure $\bar{K}$ of $K$.
Let $F \subseteq \bar{K}$ be a finite extension of $K$.
Then there exists an extension $E \subseteq \bar{K}$ of $K$
such that $E \cap F = K$ and such that,
for every abelian variety $A/K$ of positive dimension, $A(E)$ has infinite rank.
\end{prop}

\begin{proof}
Without loss of generality, we may assume that $F/K$ is Galois.
Let $\{\sigma_1, \ldots, \sigma_e\}$ be generators of $\gal(F/K)$.
We can equip $\gal(\bar{K}/K)$ with the normalized Haar measure
and consider the product measure on $\gal(\bar{K}/K)^e$.
By translation invariance, the set of all lifts of $(\sigma_1, \ldots, \sigma_e)$ in $\gal(\bar{K}/K)^e$
has the same measure as the set of lifts of $(\mathrm{id},\ldots, \mathrm{id})$.
Therefore, the set of lifts of $(\sigma_1, \ldots, \sigma_e)$ has positive measure $[F:K]^{-e}$ (\cite[Lemma 1.1]{FreyJarden}).
So by \cite[Theorem 9.1]{FreyJarden} there exists a lift
$(\tau_1,\ldots, \tau_e)$ of $(\sigma_1, \ldots, \sigma_e)$ in $\gal(\bar{K}/K)^e$ such that the following holds:
for all abelian varieties $A/K$ of positive dimension,
the rank of $A(E)$ is infinite, where $E$ is the subfield of $\bar{K}$ fixed by $\{\tau_1, \ldots, \tau_e\}$.
Every element of $F \cap E$ must be fixed by $\{\sigma_1, \ldots, \sigma_e\}$,
therefore $E \cap F = K$.
\end{proof}

Every simple abelian variety is isogenous to a simple abelian variety whose endomorphism ring is a maximal order in a division algebra:

\begin{prop}\label{maxorder-isogenyclass}
Let $A$ be an abelian variety defined over a number field $K$ and assume that $A$ is $K$-simple.
Let $R := \End_K(A)$ and $D := R \otimes_\Z \Q$.
Let $\Lambda$ be a maximal order in $D$.
There exists an abelian variety $B$ defined over $K$ which is $K$-isogenous to $A$ and such that $\End_K(B) \simeq \Lambda$.
\end{prop}

\begin{proof}
Since $R$ and $\Lambda$ are full-rank lattices in the same $\Q$-vector space,
we can take an $n \in \Z \setminus \{0\}$ such that $n \Lambda \subseteq R$.
Define
$$
	\Delta := n \Lambda
	\quad \text{and} \quad
	H := \set{T \in A(\bar{K})}{\Delta \cdot T = 0}
$$

Since $H$ is contained in $A[n]$, it is a finite group.
Consider the quotient abelian variety $B := A/H$.
Since the endomorphisms in $\Delta \subseteq R$ are defined over $K$,
it follows that $H$ is stable under $\gal(\bar{K}/K)$.
Therefore, $B$ and the projection isogeny $\pi: A \to B$
are defined over $K$.

Now we prove that the endomorphism ring of $B$ is $\Lambda$. 
Since $A$ and $B$ are $K$-isogenous, they have the same $K$-endomorphism algebra; hence, $\End_K(B)$ is an order in $D$.
Since $A(\bar{K})$ is divisible, we can write every point $P$ in $B(\bar{K})$ as $P = \pi(n \hat{P})$ for some $\hat{P} \in A(\bar{K})$. 
Let $\alpha$ be in $\Lambda$ and remark that $\alpha n = n \alpha$ belongs to $R$. Thus we can define 
$$
	\alpha P = \pi((\alpha n) \hat{P})
.$$

This definition does not depend on the choice of $\hat{P}$.
Indeed, let $P = \pi(n \hat{P}')$.
Then the difference $n(\hat{P} - \hat{P}')$ is in $H$ because $\pi$ maps it to $0$.
This implies $\Delta n(\hat{P} - \hat{P}') = 0$.
Since $\Delta$ is a right $\Lambda$-module, we have $\Delta \alpha n \subseteq \Delta n$ and so $\Delta \alpha n (\hat{P} - \hat{P}') = 0$.
This means that $\pi((\alpha n) \hat{P})=\pi((\alpha n) \hat{P}')$.

It is clear that $\alpha$ is an endomorphism
and that the above map $\Lambda \to \End_K(B)$ is an injection of rings.
Since $\Lambda$ is a maximal order, this must be an isomorphism.
\end{proof}
%===================================================================================================

\section{The condition of the support problem}\label{sec-condition}
%---------------------------------------------------------------------------------------------------

Let $A$ be an abelian variety defined over a number field $K$ and let $P$ and $Q$ be points in $A(K)$.
The support problem asks whether there exist a $K$-endomorphism of $A$ which maps $P$ to $Q$,
provided that the following condition is satisfied:

\begin{cond*}[SP]
For all but finitely many primes $\mathfrak p$ of $K$, the order of $(Q \bmod \mathfrak p)$ divides the order of $(P \bmod \mathfrak p)$.
\end{cond*}

We reformulate the condition of the support problem
by using $\End_K(A)$-modules instead of points on $A(K)$.
Let $Q$ be a point in $A(K)$ and let $\calM$ be an $\End_K(A)$-submodule of $A(K)$.
The condition of the support problem for modules is the following:

\begin{cond*}[SPM]
For all but finitely many primes $\mathfrak p$ of $K$,
the order of $(Q \bmod \mathfrak p)$ divides the exponent of $(\calM \bmod{\frakp})$.
\end{cond*}

The question now is whether $Q$ belongs to $\calM$.
%In the case where $\calM$ is generated by one element $P$,
%condition (SPM) for $\calM$ and $Q$ is equivalent to condition (SP) for $P$ and $Q$.
%However, condition (SPM) is not more general than condition (SP).
%Indeed, let $\calM$ be generated by $n$ elements $P_1,\dots,P_n$ and
%consider the points $P'=(P_1,\dots,P_n)$ and $Q'=(Q,0,\ldots,0)$ in $A^n(K)$.
%Condition (SPM) for $\calM$ and $Q$ is equivalent to Condition (SP) for $P'$ and $Q'$. 
For free modules, we have the analogue of {\cite[Proposition 9]{Perucca2}}:

\begin{theorem}\label{supp-free}
Let $A$ be an abelian variety defined over a number field $K$.
Let $\calF$ be a free $\End_K(A)$-submodule of $A(K)$ and let $Q \in A(K)$.
If $Q$ and $\calF$ satisfy Condition~(SPM), then $Q \in \calF$.
\end{theorem}
\begin{proof}
Let $\{P_1, \dots, P_n\}$ be a basis of $\calF$
and consider $P' = (P_1, \dots, P_n) \in A^n(K)$.
Since $\calF$ is a free module, the point $P'$ is independent in $A^n$.
Then we can apply \cite[Proposition 9]{Perucca2} to the points $P' = (P_1, \dots, P_n)$ and $Q' = (Q, 0, \dots, 0)$ in $A^n(K)$.
We find $Q' = \phi(P')$ for some $\phi \in \End_K(A^n)$, which implies $Q \in \calF$.
\end{proof}

It is important to note that the conditions (SP) and (SPM)
do not depend on the field:
if the condition is satisfied over a field $K$,
it is also satisfied over any finite extension.

In this paper, coherently to the other references on the support problem,
we consider the conditions (SP) and (SPM) for all but finitely many primes $\frakp$ of $K$.
However,
it is possible to require the conditions only for a set of primes $\frakp$ of $K$ of Dirichlet density $1$.
The same results hold as soon as the proofs are based on the Cebotarev Density Theorem.
For example, one has:

\begin{theorem}[{\cite[Corollary 8 and Proposition 9]{Perucca2}}] \label{classy-l}
Let $A$ and $A'$ be products of an abelian variety and a torus defined over a number field $K$.
Let $R$ be a point in $A(K)$ and let $R'$ be a point in $A'(K)$.
Let $\ell$ be a rational prime and let $S$ be a set of primes of $K$ of Dirichlet density $1$.
Suppose that for every $\frakp \in S$ we have
$$\ord_\ell(R \bmod \mathfrak p)\geq \ord_\ell(R' \bmod \mathfrak p)\,.$$
Then there exist $\phi \in \Hom_K(A, A')$ and a non-zero integer $c$
such that $\phi(R)=cR'$.
If $\Z R$ is Zariski dense in $A$, one can take $c$ coprime to $\ell$.
\end{theorem}

\begin{cor}
Let $A$ be the product of an abelian variety and a torus defined over a number field $K$.
Let $R$ be a point in $A(K)$ such that $\mathbb Z R$ is Zariski-dense in $A$.
Let $S$ be a set of primes of $K$ of Dirichlet density $1$.
\begin{enumerate}
\item The sequence
$$\{\ord(R \bmod \mathfrak p)\}_{\mathfrak p\in S}$$
determines the $K$-isomorphism class of $A$ and determines $R$ up to $K$-isomorphism.
\item Let $\ell$ be a prime number and write $\ord_\ell$ for the $\ell$-adic valuation of the order.
The sequence
$$\{\ord_\ell(R \bmod \mathfrak p)\}_{\mathfrak p\in S}$$
determines the $K$-isogeny class of $A$.
\end{enumerate}
\end{cor}
\begin{proof}
Let $A'$ be the product of an abelian variety and a torus
defined over $K$ and let $R'$ be a point in $A'(K)$ such that $\mathbb Z R'$ is Zariski-dense in $A'$.
Suppose that $\ord_\ell(R \bmod \frakp)=\ord_\ell(R' \bmod \frakp)$
for every $\frakp \in S$.
Then by Theorem~\ref{classy-l} there exist an integer $c$ coprime to $\ell$
and a $K$-homomorphism $\phi$ from $A$ to $A'$ mapping $R$ to $c R'$
and analogously there exist an integer $c'$ coprime to $\ell$
and a $K$-homomorphism $\phi'$ from $A'$ to $A$ mapping $R'$ to $c' R$.
Then $\phi' \circ \phi$ maps $R$ to $c' c R$.
Since $\Z R$ is Zariski dense in $A$, we deduce $\phi'\circ \phi=[c' c]$.
In particular, $\phi$ is an isogeny of degree coprime to $\ell$.
Now suppose that $\ord(R \bmod \frakp) = \ord(R' \bmod \frakp)$ for every $\frakp \in S$.
Then we can take $c = c' = 1$
(consider a suitable finite linear combination of the isogenies obtained for each prime $\ell$).
Thus $\phi$ is a $K$-isomorphism mapping $R$ to $R'$.
\end{proof}

%===================================================================================================

\section{Positive results for the constant of the support problem}
%---------------------------------------------------------------------------------------------------

In this section, we prove that Question~\ref{cimin}
has an affirmative answer provided that the abelian variety is of the following type:
it is the product of powers of simple abelian varieties,
which are in pairs non-isogenous and whose endomorphism rings are maximal orders in division algebras.

We start with the torsion-free case.

\begin{theorem}\label{supp-tf-max}
Let $A_1, \ldots, A_n$ be abelian varieties defined over a number field $K$
and let $A := A_1 \times \ldots \times A_n$.
Suppose that all $A_i$ are $K$-simple and that $\Hom_K(A_i,A_j)=\{0\}$ whenever $i \neq j$.
Assume that every $\End_K(A_i)$ is a maximal order.
Let $\calM$ be an $\End_K(A)$-submodule of $A(K)$ and let $Q \in A(K)$.
If $Q$ and $\calM$ satisfy Condition (SPM) and $\calM$ is torsion-free,
then $Q \in \calM$.
\end{theorem}

\begin{proof}
Let $R := \End_K(A)$.
The assumptions on $A$ imply that $R \simeq \End_K(A_1) \times \ldots \times \End_K(A_n)$.
Define $\calM_i$ as the projection of $\calM$ onto the factor $A_i(K)$.
So $\calM_i$ is an $\End_K(A_i)$-module and we have
$$
	\calM = \calM_1 \times \ldots \times \calM_n
.$$
Let $\calN := \calM + R Q$.
Analogously, we can write $\calN = \calN_1 \times \ldots \times \calN_n$.

Since every $\End_{\bar{K}}(A_i)$ is a finitely generated group,
there exists a finite extension $F$ of $K$ such that $\End_{\bar{K}}(A_i) = \End_F(A_i)$
for all $i$.
We apply Proposition~\ref{arno} to find a field $E \subseteq \bar{K}$
such that $E \cap F = K$ and such that every $A_i(E)$ has infinite rank as a $\Z$-module.
Clearly, $\End_E(A_i) = \End_K(A_i)$ for all $i$.
Recall that every non-zero element of $\End_K(A_i)$ divides a non-zero integer. Then, by applying Lemma~\ref{infinite-rank-equiv}, it easily follows that $A_i(E)$ contains a free $\End_K(A_i)$-module of infinite rank.

The modules $\calM$ and $\calN$ are finitely generated
since $A(K)$ is finitely generated.
By assumption $\calM$ is torsion-free;
we now prove that $\calN$ is torsion-free.
Let $e$ be the exponent of $\calN_\mathrm{tor}$ and suppose $e > 1$.
For all $i = 1, \ldots, n$,
Theorem~\ref{free-module} (with $G:= A_i(E)$ and $R: = \End_K(A_i)$)
shows that $\calM_i$ is contained
in a finitely generated free $\End_K(A_i)$-module $\calF_i \subseteq A_i(E)$.
For every $i$, let 
$\{F_{i 1}, \dots, F_{i r_i}\}$ be a basis for $\calF_i$.
Let $L \subseteq E$ be a finite extension of $K$ where all points $F_{i j}$ are defined.
Since the points $\{F_{i j}\}$ are independent,
by \cite[Proposition 12]{Perucca1}
there exists a positive density of primes $\frakp$ of $L$
such that for every $i$ and $j$ the order of $(F_{i j} \bmod \frakp)$ is coprime to $e$.
Hence, the exponent of $(\calM \bmod \frakp)$ is coprime to $e$.
After removing finitely many primes $\frakp$ we may assume
that $\ord(Q \bmod \frakp) \mid \exp(\calM \bmod \frakp)$
and also that $\exp(\calN_\mathrm{tor} \bmod \frakp) = e$.
It follows that the exponent of $(\calN \bmod \frakp)$ is coprime to $e$.
We have a contradiction:
the exponent of $(\calN \bmod \frakp)$ is a multiple of $e$,
but it is also coprime to $e$.

We can apply Theorem~\ref{free-module}
on $\calM_i \subseteq \calN_i$, for every $i=1,\ldots, n$.
We find that $\calM_i$ is contained in a finitely generated free $\End_K(A_i)$-module $\calF_i \subseteq A_i(E)$
such that $\calF_i \cap \calN_i = \calM_i$.
These free modules $\calF_i$ can be chosen of arbitrarily large rank,
so choose them such that their ranks are all equal to some $r>0$.
Then $\calF := \calF_1 \times \dots \times \calF_n$ is a free $R$-module of rank $r$
such that $\calF \cap \calN = \calM$.
Let $L\subseteq E$ be a finite extension of $K$ such that all points of $\calF$
are defined over $L$.
Since $Q$ and $\calF$ satisfy Condition~(SPM)
and $\calF$ is free over $\End_L(A) = R$, Theorem~\ref{supp-free} implies that $Q \in \calF$.
Thus $Q\in \calF \cap \calN = \calM$.
\end{proof}

\begin{cor}\label{supp-tf-points-max}
Let $A$ be an abelian variety defined over a number field $K$.
Suppose that $A=\prod_{i=1}^n A_i^{e_i}$ is the product of powers of $K$-simple abelian varieties,
which are in pairs non $K$-isogenous.
Suppose that $\End_K(A_i)$ is a maximal order for every $i=1,\dots,n$.
Let $P$ and $Q$ be points in $A(K)$ satisfying Condition (SP).
Suppose that the $\End_K(A)$-module generated by $P$ is torsion-free.
Then there exists $\phi$ in $\End_K(A)$ such that $\phi(P) = Q$.
\end{cor}

\begin{proof}
Let $\bar{A} = A_1 \times \ldots \times A_n$ and $\bar{\calM} = \Hom_K(A,\bar{A})\cdot P$.
The assumption on $P$ implies that $\bar{\calM}$ is a torsion-free $\End_K(\bar{A})$-module.
Notice that the identity of $\End_K (A)$ can be written as $\beta_1 \alpha_1 + \dots + \beta_m \alpha_m$
for some $m \in \N$, where $\alpha_i\in \Hom_K(A,\bar{A})$ and $\beta_i\in \Hom_K(\bar{A},A)$.

Let $\sigma$ be any element of $\Hom_K(A,\bar{A})$ and let $\bar{Q} := \sigma Q$.
If $\frakp$ is a prime of $K$, we have $\ord(\bar{Q} \bmod \frakp) \mid \ord(Q \bmod \frakp)$
and $\ord(\beta_i \alpha_i P \bmod \frakp)\mid \ord(\alpha_i P \bmod \frakp)$ for every $i=1,\ldots,m$.
We deduce that $\ord(P \bmod \frakp)\mid \exp(\bar{\calM} \bmod \frakp)$
and then that $\bar{Q}$ and $\bar{\calM}$ satisfy Condition (SPM).
By applying Theorem~\ref{supp-tf-max}, we get that $\bar{Q}\in \bar{\calM}$.

Since $\sigma$ was choosen freely,
for every $i=1,\ldots,m$ there exists $\psi_i\in \Hom(A, \bar{A})$ such that $\alpha_iQ=\psi_iP$.
Thus $Q=\sum_i \beta_i\alpha_iQ=(\sum_i \beta_i\psi_i)(P)$.
\end{proof}

We now turn our attention from torsion-free $\End_K(A)$-modules
to the general case.
Consider the following property:

\begin{define}
Let $R$ be a ring and let $\calM$ be an $R$-module
with a finite number of elements.
We call $\calM$ \emph{semi-cyclic} if the following property is satisfied:
for any two elements $T_1$ and $T_2$ in $\calM$ with $\ord(T_1) \mid \ord(T_2)$,
we must have that $T_1 = \phi T_2$ for some $\phi \in R$.
\end{define}

This notion of semi-cyclic is ``in between'' the notions
of cyclic group and cyclic module.
Indeed, a cyclic group is obviously a semi-cyclic $R$-module.
On the other hand, a semi-cyclic $R$-module is generated (as $R$-module) by any element of largest order.

Let $A$ be an abelian variety defined over a number field $K$.
Whenever the torsion part of $A(K)$ is not a semi-cyclic $\End_K(A)$-module,
the constant of the support problem is greater than $1$ (take $P$ and $Q$ independent torsion points of the same order).

\begin{cor}
Let $A$ be as in Corollary~\ref{supp-tf-points-max}.
Let $c \in \N$ be such that $c \cdot A(K)_\mathrm{tor}$ is a semi-cyclic $\End_K(A)$-module.
Let $P$ and $Q$ be points in $A(K)$ satisfying Condition~(SP).
Then there exists $\phi$ in $\End_K(A)$ such that $Q = \phi(P) + T$ for some $T \in A(K)[c]$.
\end{cor}
\begin{proof}
Let $\calM$ be the $\End_K(A)$-module generated by $P$
and let $e$ be the exponent of the torsion part of $\calM$.
By applying Corollary~\ref{supp-tf-points-max} to $e P$ and $e Q$,
we find $\phi(e P) = e Q$ for some $\phi$ in $\End_K(A)$.

Let $T := Q - \phi P$, then $\ord(T) \mid e$.
Let $T_e$ be a torsion point in $\calM$ of order $e$ and write $T_e = \tau P$.
The fact that $c A(K)_\mathrm{tor}$ is semi-cyclic
implies that $c T = \psi c T_e$ for some $\psi \in \End_K(A)$.
Now we can write $Q = (\phi + \psi \tau) P + (T - \psi T_e)$ with $c(T - \psi T_e) = 0$.
\end{proof}

This corollary has two important special cases. Firstly, 
if $A(K)_\mathrm{tor}$ is semi-cyclic then we can take $c = 1$
and we find that $Q = \phi(P)$ for some $\phi \in \End_K(A)$.
Secondly, since the zero module is semi-cyclic,
we can always take $c$ to be the exponent of $A(K)_\mathrm{tor}$.
Then we have $Q = \phi(P) + T$ for some $T \in A(K)_\mathrm{tor}$ and $\phi \in \End_K(A)$.

The question whether or not (SP) implies $Q = \phi(P) + T$ for some torsion point $T$ in $A(K)$
has been investigated by Larsen and Schoof in \cite{Larsenwhitehead} and \cite{LarsenSchoof}.
They showed in \cite{LarsenSchoof} that this is not true in general.
However in \cite[Proposition 4.3 and Theorem 5.2]{Larsenwhitehead}
Larsen proved that the above property holds for at least one variety in every $K$-isogeny class (whenever all Tate modules of $A$ are integrally semi-simple, \cite[Definition 4.1]{Larsenwhitehead}).
This also follows from our results:
by the Poincar\'e Reducibility Theorem and by Proposition~\ref{maxorder-isogenyclass},
in every $K$-isogeny class there is at least one variety satisfying the hypothesis of Corollary~\ref{supp-tf-points-max}.
%===================================================================================================

\section{Refined counterexamples to the support problem}\label{examples}
%---------------------------------------------------------------------------------------------------

\subsection{First counterexample}\label{exampleJ}
%---------------------------------------------------------------------------------------------------

We construct a counterexample to Question~\ref{cimin}.
The abelian variety in this counterexample
is the square of an absolutely simple abelian variety whose endomorphism ring is not a maximal order.

Let $\zeta_7 $ be a primitive seventh root of unity and consider $\tau := \zeta_7 + \zeta_7\inv$.
The ring $R := \Z[2 \tau, 2 \tau^2]$ is a non-maximal order in $\Q(\tau)$.
Let $\frakm := (2, 2\tau, 2\tau^2)$, a maximal ideal in $R$ with residue field $\FF{2}$.

In the Appendix, we constructed an abelian variety $A$ defined over a number field $K$
satisfying the following properties:
it is absolutely simple, it has dimension $6$, $\End_{\bar{K}}(A) = R$
and $A[2] \simeq (R/2R)^2 \times (R/\frakm)^6$ as $R$-modules.

Enlarge $K$ if necessary such that the $2$-torsion points on $A$ are $K$-rational,
such that all endomorphisms of $A$ are defined over $K$
and such that $A(K)$ contains a point of infinite order.
Let $L_1, \ldots, L_8 \in A[2]$ be such that
$$
	A[2] = (R/2R) L_1 \oplus (R/2R) L_2 \oplus (R/\frakm) L_3 \oplus \dots \oplus (R/\frakm) L_8
.$$
In the ring $R/2R$, all elements apart from $0$ and $1$ have $\frakm$ as annihilator.
This implies the following crucial fact:

\begin{obs}\label{obs-ann}
Every point of $A[2]$ has as annihilator inside $\End_K(A)$
either $(1)$, $(2)$ or $\frakm = (2, 2 \tau, 2 \tau^2)$.
\end{obs}

A result by Bogomolov (\cite[Corollaire 1]{Bogomolov}) tells us
that the image of the $2$-adic representation of $A$
contains an open subset of the homotheties of the Tate module $\mathrm{T}_2 A$.
In particular, there exists an integer $t\geq 2$
and an element of the Galois group which fixes every point in $A[2^t]$ and does not fix any point of $A$ order $2^{t+1}$.
Then, after extending $K$, we may assume that all the $2^t$-torsion points of $A$ are $K$-rational
but no point of order $2^{t+1}$ is $K$-rational.
For $i \in \{1,\dots,8\}$, choose $T_i$ in $A(K)$ such that $2^{t-1} T_i = L_i$.
In particular, $T_i$ has order $2^t$.

Let $S$ be a point of infinite order on $A(K)$.
Let $G := A \times A$ and consider the following points in $G(K)$:
$$
	P = (2 S + T_1, 2 \tau S + T_2)
	\quad\text{and}\quad
	Q = (2 \tau^2 S + T_3, 0)
.$$

We claim that the points $P$ and $Q$ satisfy Condition (SP)
for every prime $\frakp$ of good reduction for $A$, not over $2$.
By \cite[Theorem~C.1.4]{hindry-silverman},
the reduction modulo $\frakp$ gives an isomorphism from $A[2^t]$ to $(A \bmod \frakp)[2^t]$.
If $n$ is the order of $(P \bmod \frakp)$, we have 
\begin{equation}\label{Pmodp}
	2 n S + n T_1 \equiv 0 \pmod{\frakp}
	\quad \text{and} \quad
	2 \tau n S + n T_2 \equiv 0 \pmod{\frakp}
.\end{equation}
It follows that $(2 \tau n T_1\bmod\frakp) = (2 n T_2 \bmod\frakp)$.
Then we have $2 \tau n T_1 = 2 n T_2$.
This is only possible if $2 \tau n T_1 = 2 n T_2 = 0$.
Since $\Ann(L_1) = (2)$, we have $\Ann(T_1) = (2^t)$;
hence $n$ is a multiple of $2^t$.

From \eqref{Pmodp} we deduce that $(n S \bmod{\frakp})$ equals $(U \bmod \frakp)$ for a point $U$ in $A[2]$.
Since $n T_3 = 0$, to prove that $(n Q \bmod{\frakp}) = 0$, it suffices to show that $2 \tau^2 U = 0$.
By \eqref{Pmodp}, we know that $(2 \tau n S \bmod{\frakp}) = 0$.
Therefore $2 \tau U = 0$.
Because of Observation~\ref{obs-ann}, this implies $2 \tau^2 U = 0$.

Suppose that $c$ is an integer such that $\phi(P) = c Q$ for some $\phi$ in $\End_K(A)$.
We now prove that $c$ must be a multiple of $2^{t+1}$.
Because $A(K)$ has no torsion point of order $2^{t+1}$,
this will give a counterexample to Question~\ref{cimin}.

Since $\phi(P) = c Q$,
there exist $\phi_1, \phi_2 \in \End_K(A)$ such that
\begin{equation}
	\phi_1(2 S + T_1) + \phi_2(2 \tau S + T_2) = 2 \tau^2 c S + c T_3
.\end{equation}
After rearranging the terms:
\begin{equation}\label{counter-cQfP}
	(2 \tau^2 c - 2 \phi_1 - 2 \tau \phi_2) S = \phi_1 T_1 + \phi_2 T_2 - c T_3
.\end{equation}
Since $S$ has infinite order and the points $T_1$, $T_2$, $T_3$ are independent over $R/2^t R$,
\eqref{counter-cQfP} implies
\begin{equation}\label{neweqn}
	2 \tau^2 c - 2 \phi_1 - 2 \tau \phi_2=0;\quad
	\phi_1 T_1 = \phi_2 T_2 = c T_3 = 0
.\end{equation}

Since $\Ann(T_1) = \Ann(T_2) = (2^t)$ and $T_3$ has order $2^t$, we can divide $\phi_1$, $\phi_2$ and $c$ by $2^t$.
So there exist $\phi_1'$ and $\phi_2'$ in $\End_K(A)$
and an integer $c'$ such that $\phi_1 = 2^t \phi_1'$, $\phi_2 = 2^t \phi_2'$ and $c = 2^t c'$.
Then \eqref{neweqn} implies
$$
	2 \tau^2 c' - 2 \phi_1' - 2 \tau \phi_2' = 0
.$$

An odd multiple of $2 \tau^2$ is not contained in the ideal $(2, 2\tau)$.
Therefore, $c'$ is even and $c$ is divisible by $2^{t+1}$.
%===================================================================================================

\subsection{Second counterexample}\label{exampleA}
%---------------------------------------------------------------------------------------------------

We construct a different counterexample for Question~\ref{cimin}.
Only after a finite extension of the base field,
the abelian variety of this counterexample is isomorphic
to the power of an elliptic curve whose endomorphism ring is a maximal order.
The two given points lie on a proper abelian subvariety
and one point is the image of the other by an isomorphism of the subvariety.

Consider the following elliptic curves over $\mathbb Q$:

\begin{align*}
	A: y^2 & =x^3+40 \\
	B: y^2 & =x^3+5
\end{align*}

Both $A$ and $B$ have complex multiplication over the field $\Q(\zeta_3)$,
where $\zeta_3$ corresponds to the map $x\mapsto \zeta_3x$; $y \mapsto y$.
This means that $\End_{\bar{\Q}}(A) \simeq \End_{\bar{\Q}}(B) \simeq \Z[\zeta_3]$,
the maximal order in $\Q(\zeta_3)$.
The two curves are isomorphic over $\Q(\sqrt{2})$
(from $B$ to $A$, consider $\theta$: $x\mapsto 2x$; $y \mapsto 2\sqrt{2}y$).
However, the two curves are not isogenous over $\Q$.
Suppose there is an isogeny $\alpha \in \Hom_\Q(A,B)$.
Then $\theta \circ \alpha$ is an endomorphism of $A$ defined over $\Q(\sqrt{2})$.
Since $\End_{\Q(\sqrt{2})}(A) = \End_\Q(A)$ and $\alpha$ is defined over $\Q$,
it would follow that $\theta$ is also defined over $\Q$.

The map $\theta$ induces an isomorphism of Galois modules from $B[2]$ to $A[2]$.
Thus the group
$$
	H := \set{(\theta T, T) \in (A(\bar{K}), B(\bar{K}))}{T \in B[2]}
$$ is $\gal(\bar{\Q}/\Q)$-stable.
It follows that the abelian variety
$$G=((A{\times} B)\,/H)\times B$$
is defined over $\mathbb Q$.
We claim that $G[2](\mathbb Q)$ is zero.
Since $B$ has no torsion over $\Q$, it suffices to show that $(A\times B)/ H$ has no $2$-torsion over $\Q$.
Over $\Q(\sqrt{2})$, we have
\begin{multline*}
	(A\times B)/H = (A\times B)/\{(\theta T, T)\} \simeq (B\times B)/\{(T,T)\} \\
		\simeq (B\times B)/\{(T,0)\} \simeq B/B[2] \times B \simeq B \times B
.\end{multline*}
Since $B[2](\Q(\sqrt{2})) = \{0\}$, it follows that $(A\times B)/H$ has no $2$-torsion over $\Q$.

The point $R=(-1,2)$ in $B(\mathbb Q)$ has infinite order. Define the following points in $G(\mathbb Q)$:
$$P=([(0,R)] , 0)\;;\quad Q=([(0,0)] , R)\,.$$
The two points belong to the abelian subvariety $(\{0\}{\times} B)/H \times B \simeq B \times B$.
The isomorphism which switches the two factors maps $P$ to $Q$.
In particular, the points $P$ and $Q$ satisfy Condition~(SP).

We now prove that the above points provide a counterexample to Question~\ref{cimin}.
Since $G[2](\mathbb Q)$ is zero, it suffices to show that if $\phi(P)=cQ$ for some $\phi$ in $\End_{\Q}(G)$
and $c\in \mathbb Z$, then $c$ must be even.

Let $\Phi$ be the composition
$$A\times B\stackrel{\pi}{\longrightarrow} (A\times B)/H \stackrel{\iota}{\hookrightarrow} G\stackrel{\phi}{\longrightarrow}G\stackrel{\pi_B}{\longrightarrow} B$$
where $\pi$ is the quotient map, $\iota$ is the inclusion and $\pi_B$ is the projection of $G$ onto its direct factor $B$.

Having $\Phi: A \times B \to B$ is equivalent to having
$\Phi_A$ in $\Hom_{\mathbb Q}(A,B)$ and $\Phi_B$ in $\End_\Q(B)$.
We know that $\Phi_A$ is zero since $A$ and $B$ are not $\mathbb Q$-isogenous.
Since $\phi(P)=cQ$, we deduce that $\Phi_B(R) = c R$;
hence $\Phi_B$ is the multiplication by $c$.
Let $(\theta T, T)$ be a non-zero element of $H$. Notice that the order of $\theta T$ equals the order of $T$.
If $c$ is odd, we have a contradiction:
$$0=\Phi_A(\theta T)=\Phi(\theta T,0)=\Phi(0, -T)=\Phi_B(-T)=-cT\neq 0.$$
%===================================================================================================

\subsection{The support problem for tori}\label{exampleAtori}
%---------------------------------------------------------------------------------------------------

Let $G$ be a torus defined over a number field $K$ and let $P$ and $Q$
be points in $G(K)$ satisfying condition (SP). The support problem asks whether there exists a $K$-endomorphism of $G$ mapping $P$ to $Q$.

If $G$ is one-dimensional then $\phi(P)=Q$ for some $\phi$ in $\End_K(G)$, as follows from a result by Khare (\cite[Proposition 3]{Kharegalois}).
The second author proved in \cite[Proposition 12]{Perucca2} that if $G$ is split then $\phi(P)=Q$ for some $\phi$ in $\End_K (G)$. Furthermore, she proved that in general $\phi(P)=dQ$ for some $\phi$ in $\End_K G$, where $d$ is the degree of the smallest Galois extension of $K$ splitting the torus (\cite[Lemma 2 and Proposition 12]{Perucca2}). We now answer in the negative the question of the support problem for tori.

Let $T$ be any torus satisfying the following property: the intersection $T_a\cap T_d$ of the maximal anisotropic subtorus with the maximal split subtorus is non-trivial (this intersection is always finite, see \cite[Chapter III, Section 8.15, Proposition]{Borel}). Let $R$ be a point in $T_d(K)$ which is independent: this amounts to choosing some multiplicatively independent elements in $K^*$. Then a counterexample is given by:
$$G=T\times T_d\quad;\quad P=(R,0)\quad;\quad Q=(0,R)\,.$$
It is clear that the points $P$ and $Q$ satisfy Condition (SP).
By \cite[Main Theorem]{Perucca2}, we have $\phi(P)=cQ$ for some minimal positive integer $c$ and for some $\phi$ in $\End_K G$.
Let $\Phi$ be the composition
$$T\stackrel{\iota}{\hookrightarrow} G\stackrel{\phi}{\longrightarrow}G\stackrel{\pi}{\longrightarrow} T_d$$
where $\iota$ is the inclusion and $\pi$ is the projection of $G$ onto $T_d$.
Since $\Phi(R)=cR$ and $R$ is independent in $T_d$, the restriction of $\Phi$ to $T_d$ is the multiplication by $c$. Because $\Hom_K(T_a, T_d)$ is zero \cite[Chapter III, Section 8.15, Proposition]{Borel}, the restriction of $\Phi$ to $T_a$ is zero. We deduce that the points in $(T_a\cap T_d)(\bar{K})$ are killed by the multiplication by $c$. So $c$ must be a multiple of the exponent of the group $(T_a\cap T_d)(\bar{K})$ and in particular it is not $1$.

Since $c$ divides the degree of the smallest Galois extension
of $K$ splitting the torus (\cite[Lemma 2 and Proposition 12]{Perucca2}), we also provided an alternative proof of the following:
for every torus $T$ defined over a number field $K$,
the exponent of $(T_a\cap T_d)(\bar{K})$ divides the degree of the smallest Galois extension of $K$ where $T$ splits.
%===================================================================================================

\section{Appendix}
%---------------------------------------------------------------------------------------------------

In this appendix, we construct an abelian variety for the counterexample in section \ref{exampleJ}.

Let $\zeta_7 $ be a primitive seventh root of unity and consider $\tau := \zeta_7 + \zeta_7\inv$.
The minimal polynomial of $\tau$ is  $x^3 + x^2 - 2 x - 1$.
The number field $\Q(\tau)$ is totally real and Galois with ring of integers $\Z[\tau]$.
The ring $R := \Z[2 \tau, 2 \tau^2]$ is a non-maximal order in $\Q(\tau)$.
Let $\frakm := (2, 2\tau, 2\tau^2)$ be a maximal ideal in $R$ with residue field $\FF{2}$.

\begin{theorem}\label{th-appendix}
There exists a number field $K$ and an abelian variety $A$ defined over $K$
which is absolutely simple of dimension $6$,
with $\End_{\bar{K}}(A) = R$
and such that $A[2] \simeq (R/2R)^2 \times (R/\frakm)^6$ as $R$-modules.
\end{theorem}

The outline of the construction is the following: %We first work over the field $\C$.
 We define a lattice $\Lambda$ in $\C^6$, by which we mean a discrete subgroup of $\C^6$ of rank $12$.
We show that the complex torus $\mathcal A := \C^6/\Lambda$ is an abelian variety
by exhibiting a positive definite hermitian form on $\C^6$
whose imaginary part takes integer values on $\Lambda\times \Lambda$. We check that $\End_{\C}(\mathcal A)$ is $R$ and that $\mathcal A[2] \simeq (R/2R)^2 \times (R/\frakm)^6$ as $R$-modules.
We conclude by applying a result on the specialization of abelian varieties. 

\begin{proof}[Proof of Theorem~\ref{th-appendix}]
Consider the following matrices, where $\mathbf 0_{3}$ denotes the zero matrix of dimension $3$ by $3$:
$$
	M = \begin{pmatrix}
		0 & 0 & 1 &  &  &  \\
		1 & 0 & 2 &  &\mathbf 0_{3}  &  \\
		0 & 1 & -1 &  &  &  \\
		 &  &  & 0 & 0 & 1 \\
		 &  \mathbf 0_{3} &  & 1 & 0 & 2 \\
		 &  &  & 0 & 1 & -1
	\end{pmatrix}
	;\quad
	X = \begin{pmatrix}
		2 & -1 & 2 &  &  &  \\
		-1 & 2 & -2 & & \mathbf 0_{3}  &  \\
		2 & -2 & 5 &  &  &  \\
		 &  &  & 2 & -1 & 2 \\
		 &  \mathbf 0_{3}&  & -1 & 2 & -2 \\
		 &  &  & 2 & -2 & 5
	\end{pmatrix}
.$$
These matrices satisfy $X M= M^T X$.
The minimal polynomial of $M$ is $x^3 + x^2 - 2 x - 1$.
This implies that $\Z[M] = \set{a_0 \mathrm{I}_6 + a_1 M + a_2 M^2}{a_0, a_1, a_2 \in \Z}$
is a commutative integral domain isomorphic to $\Z[\tau]$.
The characteristic polynomial of $X$ is $(x-1)^4 (x-7)^2$ so in particular $X$ is positive definite.

For $i = 1, \dots, 6$ we call $\vece_i$ the column vector that has only one non-zero entry,
located at the $i$-th row and of value $1$.
Notice that we have 
$$
	\vece_2 = M \vece_1;\quad
	\vece_3 = M^2 \vece_1;\quad
	\vece_5 = M \vece_4;\quad
	\vece_6 = M^2 \vece_4
.$$
Let $\alpha_1, \dots, \alpha_9$
be real numbers such that $\{1, \alpha_1, \dots, \alpha_9\}$ is a $\Q$-linearly independent set.
Let $\omega$ be a positive real number such that
$\omega^2$ is not equal to $f(\alpha_1, \dots, \alpha_9)$
for any polynomial $f \in \Q[x_1, \dots, x_9]$ of degree at most $2$. Write
$$
	\vecr = \begin{pmatrix}\alpha_1 \\ \alpha_2 \\ \alpha_3 \\ \alpha_4 \\ \alpha_5 \\ \alpha_6 \end{pmatrix};\quad
	\vecs = \begin{pmatrix}\alpha_4 \\ \alpha_5 \\ \alpha_6 \\ \alpha_7 \\ \alpha_8 \\ \alpha_9 \end{pmatrix}
$$
and define
\begin{align*}
	\vece_7 &= \vecr + (\omega \mi) \vece_1; &
	\vece_{10} &= \vecs + (\omega \mi) \vece_4; \\
	\vece_8 &= 2 M \vece_7 = 2 M \vecr + (2\omega \mi)\vece_2; &
	\vece_{11} &= 2 M \vece_{10} = 2 M \vecs + (2\omega \mi)\vece_5; \\
	\vece_9 &= 2 M^2 \vece_7 = 2 M^2 \vecr + (2\omega \mi)\vece_3; &
	\vece_{12} &= 2 M^2 \vece_{10} = 2 M^2 \vecs + (2\omega \mi)\vece_6
.\end{align*}

By the choice of the $\alpha_i$'s and of $\omega$, the vectors $\vece_1,\ldots, \vece_{12}$ are $\R$-linearly independent.
We define $\Lambda$ to be the $\Z$-span of $\{\vece_1,\ldots, \vece_{12}\}$ inside $\C^6$. 

Consider the following positive definite hermitian form on $\C^6$:
$$
	H(\vecx, \vecy) = (\mathbf{\bar{x}}^T X \vecy) \omega\inv\,.
$$
Let $E$ be the imaginary part of $H$,
which is an $\R$-bilinear alternating form on $\C^6$.
Since  $X M = M^T X$, we have $E(M \vecx, \vecy) = E(\vecx, M \vecy)$.
Using this property, one can easily check that $E(\vecx, \vecy) \in \Z$
for all $\vecx$ and $\vecy$ in $\Lambda$.
Thus the complex torus $\mathcal A := \C^6/\Lambda$ is an abelian variety of dimension $6$ (\cite[Corollary~p.~35]{mumford}).

We can write $\Lambda = \Z^6 + \Omega \Z^6$, where $\Omega = (\vece_7 | \vece_8 | \vece_9 | \vece_{10} | \vece_{11} | \vece_{12})$.
The imaginary part of $\Omega$ is 
$$\Im(\Omega) = (\vece_1 | 2\vece_2 | 2\vece_3 | \vece_{4} | 2\vece_{5} | 2\vece_{6})\omega\,.$$
The real part $\Re(\Omega)$ is a matrix whose entries are linear combinations of the $\alpha_i$'s.

The $\C$-endomorphisms of $\mathcal A$ are the $\C$-linear maps
$\sigma: \C^6 \to \C^6$ such that $\sigma(\Lambda) \subseteq \Lambda$.
Let $S$ be a $6 \times 6$ matrix over $\C$ defining an endomorphism. % of $\mathcal A$.
We first show that $S$ has integer coefficients, and then that it belongs to $\Z[2 M, 2 M^2]$.

Since $S$ maps $\vece_1, \dots, \vece_6$
into $\Lambda$, there exist two $6 \times 6$ matrices $A_1$ and $A_2$ with coefficients in $\Z$ such that $S = A_1 + \Omega A_2$.
Similarly, $\vece_7, \dots, \vece_{12}$ get mapped into $\Lambda$ so we have integer matrices $B_1$ and $B_2$ such that
$S \Omega = B_1 + \Omega B_2$.
By equating the two ways of writing $S \Omega$, we get
$$
	\Omega A_2 \Omega + (A_1 \Omega - \Omega B_2) - B_1 = 0 
$$
Taking the real part of the above equation yields:
\begin{multline*}
	\Re(\Omega) A_2 \Re(\Omega) + (A_1 \Re(\Omega) - \Re(\Omega) B_2) - B_1 \\
	\begin{aligned}
		= &\Im(\Omega)A_2\Im(\Omega) \\
		= &(\vece_1 | 2\vece_2 | 2\vece_3 | \vece_{4} | 2\vece_{5} | 2\vece_{6})
			A_2 (\vece_1 | 2\vece_2 | 2\vece_3 | \vece_{4} | 2\vece_{5} | 2\vece_{6})
			\omega^2
	.\end{aligned}
\end{multline*}
The assumption on $\omega^2$ then implies that $A_2 = 0$.
Hence $S$ is a matrix with integer coefficients.

Since $S$ maps $\vece_7$ into $\Lambda$, there exist $c_1, \dots, c_{12} \in \Z$ such that
$S \vece_7 = \sum_{i=1}^{12} c_i \vece_i$.
Define
\begin{align*}
	C_1 &:= c_1 \mathrm{I}_6 + c_2 M + c_3 M^2; &
	C_4 &:= c_4 \mathrm{I}_6 + c_5 M + c_6 M^2 \\
	C_7 &:= c_7 \mathrm{I}_6 + 2 c_8 M + 2 c_9 M^2; &
	C_{10} &:= c_{10} \mathrm{I}_6 + 2 c_{11} M + 2 c_{12} M^2
\end{align*}
Then we have
$$
	S \vece_7 = C_1 \vece_1 + C_4 \vece_4 + C_7 \vece_{7} + C_{10} \vece_{10}
.$$

The entries of $\Re(S \vece_7)$ and of $\Re(C_7 \vece_{7} + C_{10} \vece_{10})$ are linear combinations of $\alpha_1,\ldots, \alpha_9$.
Thus the entries of $(C_1 \vece_1 + C_4 \vece_4)$, which are integers, must be all zero.
So we have $\Re((S-C_7) \vece_7) = \Re(C_{10} \vece_{10})$.
By looking at the sixth entry, we get that a linear combination of $\alpha_1,\ldots, \alpha_6$ is equal to 
$$
	2 c_{12}\alpha_7 + (2 c_{11} - 2 c_{12})\alpha_8 + (c_{10} + 2 c_{11} + 6 c_{12})\alpha_9
.$$
By the independence of the $\alpha_i$'s, we deduce that
$c_{10} = c_{11} = c_{12} = 0$; hence $C_{10} = 0$.
So we have $S \vece_7 = C_7 \vece_7$.
Again by the independence of the $\alpha_i$'s, we deduce that $S = C_7$.
This means that $S$ belongs to $\Z[2 M, 2 M^2]$.

It can easily be checked that every element in $\Z[2 M, 2 M^2]$
defines an endomorphism of $\mathcal A$.
Since $\Z[2 M, 2 M^2] \simeq \Z[2 \tau, 2 \tau^2]$ as rings,
we conclude that $\End_{\C}(\mathcal A)=R$.

We now study the action of $\End_\C(\mathcal A)$ on $\mathcal A[2]$.
Define $T_i := \vece_i/2 + \Lambda$ for all $i = 1, \dots, 12$. Since $\mathcal A[2] = (\frac{1}{2}\Lambda)/\Lambda$, it is clear that
$$
	\mathcal A[2] \simeq \bigoplus_{i=1}^{12} (\Z/2\Z) T_i
.$$
For all $i = 1, \ldots, 6$, we have $2 M T_i = 0$
and also $2 M^2 T_i = 0$.
It follows that $(\Z/2\Z) T_i$ is an $R$-module isomorphic to $R/\frakm$.
On the other hand, we have $2 M T_7 = T_8$ and $2 M^2 T_7 = T_9$.
This implies that $\bigoplus_{i=7}^{9} (\Z/2\Z) T_i$
is an $R$-module isomorphic to $R/2R$.
Similarly for $\bigoplus_{i=10}^{12} (\Z/2\Z) T_i$.

Let $F = \Q(z_1, \dots, z_s)$ be a finitely generated subfield of $\C$
such that $\mathcal{A}$ is defined over $F$,
all $2$-torsion points of $\mathcal{A}$ are $F$-rational
and $\End_F(\mathcal A)=\End_\C(\mathcal A)$.
Call $k$ the relative algebraic closure of $\mathbb Q$ in $F$, which is a number field.
Choose an affine variety $V$ over $k$ whose function field is $F$.
We can specialize $\mathcal A$ with respect to the $\bar{k}$-points of $V$.
After replacing $V$ by an open affine subvariety,
we may assume that the specialization is injective on the finite set $\mathcal A[2]$
and that the dimension of the specialized variety is $\dim_\C \mathcal A=6$.
By \cite[Theorem, Section 1]{Masser} there exists a specialization $A$ of $\mathcal A$
which is an abelian variety over a number field $K\supseteq k$
such that $\End_{\bar{K}}(A)=\End_\C(\mathcal A)=R$.
Since $\mathcal A[2]$ is mapped injectively into $A[2]$, they are isomorphic as $R$-modules.
Finally, $A$ is absolutely simple by the Poincar\'e Reducibility Theorem because $R$ has no zero divisors.
\end{proof}
%===================================================================================================

\bibliographystyle{amsplain}
\bibliography{all}
%---------------------------------------------------------------------------------------------------

\end{document}